\documentclass[12pt]{amsart}
\usepackage{etex}

\usepackage{amsthm,amssymb,enumitem,mathtools}

\usepackage{tikz}

\usetikzlibrary{arrows,shapes,automata,backgrounds,decorations,petri,positioning}

\tikzstyle{edge} = [fill,opacity=.5,fill opacity=.5,line cap=round, line join=round, line width=50pt]

\pgfdeclarelayer{background}
\pgfsetlayers{background,main}

\usepackage[margin=1in,letterpaper,portrait]{geometry}

\usepackage{amsthm}

\usepackage[latin1]{inputenc}
\usepackage{subfigure}
\usepackage{color}
\usepackage{amsmath}
\usepackage{amsthm}
\usepackage{amstext}
\usepackage{amssymb}
\usepackage{amsfonts}
\usepackage{graphicx}
\usepackage{young}
\usepackage{multicol}
\usepackage{mathrsfs}
\usepackage[all]{xy}
\usepackage{tikz}
\usepackage{mathtools}
\usepackage{tabularx}
\usepackage{array}
\usepackage{commath}
\usepackage{ytableau}

\usepackage{scrextend}

\usepackage{hyperref}

\theoremstyle{plain}
\theoremstyle{definition}
\newtheorem{theorem}{Theorem}[section]

\newtheorem{lemma}[theorem]{Lemma}

\newtheorem{definition}[theorem]{Definition}

\newtheorem{question}[theorem]{Question}
\newtheorem{example}[theorem]{Example}
\newtheorem{proposition}[theorem]{Proposition}
\newtheorem{corollary}[theorem]{Corollary}

\DeclareMathAlphabet{\mathpzc}{OT1}{pzc}{m}{it}

\newcommand{\poi}[1]{\Lambda_{#1}}
\newcommand{\weakpoi}[1]{\poi{#1}^{\text{wk}}}
\newcommand{\weakint}{_{\text{wk}}}
\newcommand{\ul}[1]{\underline{#1}}
\newcommand{\weak}{\le_{\text{wk}}}
\newcommand{\symm}{\mathfrak{S}}
\newcommand{\s}{\sigma}

\begin{document}

\title{Interval structures in the Bruhat and weak orders}%

\date{}

\author{Bridget Eileen Tenner}
\address{Department of Mathematical Sciences, DePaul University, Chicago, IL, USA}
\email{bridget@math.depaul.edu}
\thanks{Research partially supported by Simons Foundation Collaboration Grant for Mathematicians 277603 and by a University Research Council Competitive Research Leave from DePaul University.}



\begin{abstract}%
We study the appearance of notable interval structures---lattices, modular lattices, distributive lattices, and boolean lattices---in both the Bruhat and weak orders of Coxeter groups. We collect and expand upon known results for principal order ideals, including pattern characterizations and enumerations for the symmetric group. This segues naturally into a similar analysis for arbitrary intervals, although the results are less characterizing for the Bruhat order at this generality. In counterpoint, however, we obtain a full characterization for intervals starting at rank one in the symmetric group, for each of the four structure types, in each of the two posets. Each category can be enumerated, with intriguing connections to Fibonacci and Catalan numbers. We conclude with suggestions for further directions and questions, including an interesting analysis of the intervals formed between a permutation and each generator in its support.
\end{abstract}

\maketitle

The Bruhat order of a Coxeter group is a natural and appealing partial ordering on an important mathematical object. Despite that, the structure of its intervals has notable and enigmatic complexity. For example, topological properties are discussed in \cite[\S2.7]{bjorner brenti}, Dyer showed that there are only finitely many isomorphism classes of intervals of a given length in finite Coxeter groups \cite{dyer}, we previously compared generic intervals to principal order ideals in the symmetric group in \cite{tenner intervals factors}, and Bj\"orner and Ekedahl study Betti numbers related to these intervals as well their chain decompositions \cite{bjorner ekedahl}. The possible structures of principal order ideals in these posets are quite a narrow subset of the possible intervals that might appear, and even those do not always have some of the structural properties one might hope for in a poset. The weak order of a Coxeter group is a similarly important and intriguing partial ordering, with important structural results shown by Stembridge \cite{stembridge}.

In this paper, we look at these fascinating architectures and pick out the intervals that are the most well-behaved: lattices, modular lattices, distributive lattices, and boolean lattices. In previous work, we described boolean principal order ideals in the Bruhat order \cite{tenner patt-bru}. Somewhat wonderfully, those ideals can be described in terms of pattern avoidance in the symmetric group.

Here we look, more generally, at when an arbitrary interval might be a well-behaved lattice in these posets. The potential intricacies of Coxeter group elements mean that this can be highly element-specific. We begin by collecting and expanding upon known results for principal order ideals in each of these contexts. Furthermore, we focus in on the symmetric group for pattern characterizations and enumerations. These analyses of principal order ideals lead to analogous questions about the more general setting of arbitrary intervals. Indeed, we can describe the intervals that fit our four criteria in each poset, up to a point. The weak order is particularly amenable, due to a result of Stembridge \cite{stembridge}. The question for the Bruhat order, on the other hand, can be answered to some extent, but does not resolve to a clear characterization.

We devote the remainder of this work to intervals in $\symm_n$, in both posets, in which the minimum element is an atom. Remarkably, we can completely characterize the desired intervals, with each well-behaved lattice structure and in each poset. Moreover, the numbers of such intervals can be computed every time, with elegant results. The required properties for these intervals---which are so close to being principal order ideals---to be well-behaved lattices are reminiscent of the rules for principal order ideals that begin this work. However, as can be seen by comparing Tables~\ref{table:poi lattice enumerations in Sn} and~\ref{table:atom lattice enumerations in Sn}, the possibilities themselves are quite different.

The paper is organized as follows. Section~\ref{sec:preliminaries} lays out the primary objects and notation of this work. In Section~\ref{sec:poi lattices}, we show the variety of principal order ideal structures in the Bruhat and weak orders, and classify when the well-behaved ones appear (summarized in Table~\ref{table:poi lattices}). In the case of the symmetric group, we characterize these phenomena by pattern avoidance and provide enumerations for each case, in each order. This section collects and expands on previous work. Section~\ref{sec:boolean review} will briefly review the boolean-related results of \cite{tenner patt-bru} and other works. Section~\ref{sec:interval hierarchy} shows that, in an important sense, the hierarchical structure of principal order ideals lays the groundwork for arbitrary intervals in each of the Bruhat and weak orders, highlighted in Theorem~\ref{thm:bruhat intervals modular=boolean} and Corollary~\ref{cor:weak interval characterization}. In Section~\ref{sec:intervals over atoms in Sn}, we look at the special case of intervals whose minimum elements are atoms, in both the Bruhat and weak orders. In each of these settings, we explicitly characterize all such intervals that are lattices, modular lattices, distributive lattices, and boolean lattices (including Theorems~\ref{thm:boolean interval at rank 1} and~\ref{thm:lattice interval at rank 1}). We enumerate each variation, with appealing connections to Fibonacci and Catalan numbers (Theorems~\ref{thm:boolean interval at rank 1 count} and~\ref{thm:counting dist/mod lattices above atoms in the weak order}, Propositions~\ref{prop:counting lattices above atoms in the weak order} and~\ref{prop:counting boolean lattices above atoms in the weak order}, and Corollary~\ref{cor:lattice interval at rank 1 count}). We conclude the paper with a sampling of related further directions and questions in Section~\ref{sec:open questions}, including characterization and enumeration of permutations that form boolean intervals over all elements of their support, in both the Bruhat and weak orders (Corollaries~\ref{cor:boolean for all support bruhat} and~\ref{cor:boolean for all support weak}).

\section{Preliminaries}\label{sec:preliminaries}

In preparation for our main work, we use this section to highlight relevant terminology and to set notation. This effort falls into three categories---Coxeter-theoretic, poset-theoretic, and pattern-based. To avoid suggesting disproportionate importance to this material via word count, we use examples to remind the reader of key definitions, and outsource a more thorough background to texts such as \cite{bjorner brenti, kitaev, ec1}.

\subsection{Poset-theoretic terminology and notation}\label{subsec:poset defns}

We will be concerned with posets whose elements are organized in particular ways. Our motivating focus is the most demanding of these organizations (\emph{boolean} posets), but it is illuminating to consider them in a broader context and we will look at \emph{lattices}, \emph{modular lattices}, and \emph{distributive lattices}, as well. We present Figure~\ref{fig:poset examples} as a nudge toward recalling definitions of the latter three of these. In fact, the posets depicted in Figures~\ref{fig:poset examples}(bc) are characterizing features of distributive lattices: a lattice is distributive if and only if it has no sublattice isomorphic to either of those examples \cite[Theorem 4.10]{davey priestley}.

\begin{figure}[htbp]
\begin{tabular}{m{1in} m{1in} m{1.125in} m{.75in}}
(a) & (b) & (c) & (d)\\
\ \ \begin{tikzpicture}
\foreach \x in {(0,0),(0,1),(1,0),(1,1)} {\fill \x circle (2pt);}
\draw (0,0) -- (0,1) -- (1,0) -- (1,1) -- (0,0);
\end{tikzpicture}
&
\ \ \begin{tikzpicture}
\foreach \x in {(0,.5),(.5,1.5),(0,2.5),(-.5,1),(-.5,2)} {\fill \x circle (2pt);}
\draw (0,.5) -- (.5,1.5) -- (0,2.5) -- (-.5,2) -- (-.5,1) -- (0,.5);
\end{tikzpicture}
&
\ \ \begin{tikzpicture}
\foreach \x in {(0,0),(-.75,.75), (0,.75), (0,1.5), (.75,.75)} {\fill \x circle (2pt);}
\draw (0,0) -- (-.75,.75) -- (0,1.5) -- (.75,.75) -- (0,0) -- (0,1.5);
\end{tikzpicture}
&
\ \ \begin{tikzpicture}
\foreach \x in {(0,0),(.75,.75),(0,1.5),(-.75,.75),(0,2.25)} {\fill \x circle (2pt);}
\draw (0,0) -- (.75,.75) -- (0,1.5) -- (0,2.25);
\draw (0,0) -- (-.75,.75) -- (0,1.5);
\end{tikzpicture}
\end{tabular}
\caption[caption]{(a) A poset that is not a lattice.\\\hspace{\textwidth}
\phantom{\textsc{Figure 2. }}(b) A lattice that is not modular.\\\hspace{\textwidth}
\phantom{\textsc{Figure 2. }}(c) A modular lattice that is not distributive.\\\hspace{\textwidth}
\phantom{\textsc{Figure 2. }}(d) A distributive lattice that is not boolean.
}\label{fig:poset examples}
\end{figure}
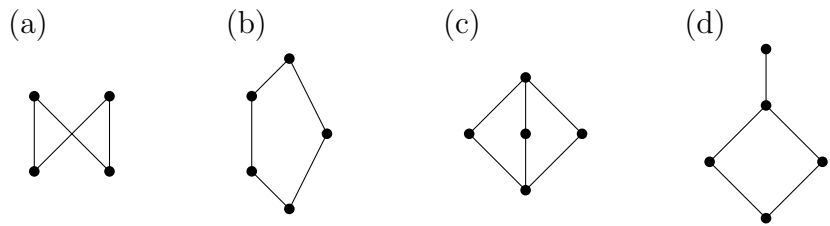

The subclass of distributive lattices that we consider here is a particularly tame family of posets.

\begin{definition}\label{defn:boolean poset}
A poset is \emph{boolean} if it is isomorphic to the poset of subsets of a finite set $S$, ordered by inclusion. We say that such a boolean poset is \emph{on $|S|$ elements}.
\end{definition}

A boolean poset is a distributive lattice, where the join operation is set intersection and the meet operation is union. Because boolean posets are unique up to isomorphism, we may refer to ``the'' boolean poset of a given size.

\begin{example}\label{ex:boolean on 4 elements}
The poset depicted in Figure~\ref{fig:poset examples}(d) is a distributive lattice, but not a boolean one. The boolean poset on four elements appears in Figure~\ref{fig:boolean poset}.
\end{example}

\begin{figure}[htbp]
\begin{tikzpicture}
\coordinate (empty) at (0,0);
\coordinate (a) at (-1.5,1);
\coordinate (b) at (-.5,1);
\coordinate (c) at (.5,1);
\coordinate (d) at (1.5,1);
\coordinate (ab) at (-2.5,2);
\coordinate (ac) at (-1.5,2);
\coordinate (ad) at (-.5,2);
\coordinate (bc) at (.5,2);
\coordinate (bd) at (1.5,2);
\coordinate (cd) at (2.5,2);
\coordinate (abc) at (-1.5,3);
\coordinate (abd) at (-.5,3);
\coordinate (acd) at (.5,3);
\coordinate (bcd) at (1.5,3);
\coordinate (abcd) at (0,4);
\foreach \x in {empty,a,b,c,d,ab,ac,ad,bc,bd,cd,abc,abd,acd,bcd,abcd} {\fill (\x) circle (2pt);}
\foreach \x in {a,b,c,d} {\draw (empty) -- (\x);}
\foreach \x in {ab,ac,ad} {\draw (a) -- (\x);}
\foreach \x in {ab,bc,bd} {\draw (b) -- (\x);}
\foreach \x in {ac,bc,cd} {\draw (c) -- (\x);}
\foreach \x in {ad,bd,cd} {\draw (d) -- (\x);}
\foreach \x in {ab,ac,bc} {\draw (abc) -- (\x);}
\foreach \x in {ab,ad,bd} {\draw (abd) -- (\x);}
\foreach \x in {ac,ad,cd} {\draw (acd) -- (\x);}
\foreach \x in {bc,bd,cd} {\draw (bcd) -- (\x);}
\foreach \x in {abc,abd,acd,bcd} {\draw (abcd) -- (\x);}
\end{tikzpicture}
\caption{The boolean poset on four elements.}\label{fig:boolean poset}
\end{figure}
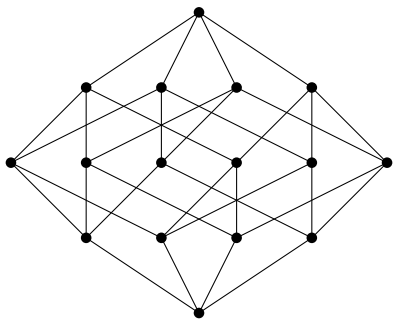

These poset categories obey the following hierarchy:
\begin{equation}\label{eqn:hierarchy}
\Big\{\text{boolean lattices}\Big\} \subset \Big\{\text{distributive lattices}\Big\} \subset \Big\{\text{modular lattices}\Big\} \subset \Big\{\text{lattices}\Big\}.
\end{equation}

\subsection{Coxeter-theoretic terminology and notation}
We now briefly give relevant Coxeter-theoretic definitions and notation, and the reader is referred to \cite{bjorner brenti} for more information.

\begin{definition}\label{defn:coxeter group}
A \emph{Coxeter group} consists of a collection $S$ of generators, all of which are involutions, and relations of the form
$$(st)^{m(s,t)} = 1,$$
where $m(s,t) = m(t,s) \in \mathbb{Z}^+ \cup \{\infty\}$, for all $s,t \in S$. When $m(s,t) = 2$, the relation $st = ts$ is a \emph{commutation}. When $m(s,t) > 2$, the relation $\underbrace{sts\cdots}_{m(s,t)} = \underbrace{tst\cdots}_{m(s,t)}$ is a \emph{braid}.
\end{definition}

As discussed in \cite[Ch.~1 and 8]{bjorner brenti}, the finite Coxeter groups of types $A$, $B$, and $D$ have combinatorial interpretations as permutations, signed permutations, and signed permutations with restriction. The finite Coxeter group of type $A_{n-1}$, denoted $\symm_n$, is the \emph{symmetric group}. The finite Coxeter group of type $B_n$, denoted $\symm^B_n$, is the \emph{hyperoctahedral group}. When discussing signed permutations, we may write $\ul{i} := -i$ for readability.

\begin{definition}
For $i \ge 1$, let $\s_i$ be the map swapping $i$ and $i+1$ (and, when relevant, swapping $-i$ and $-(i+1)$) and fixing all other letters. Let $\s_0$ be the map swapping $1$ and $-1$ and fixing all other letters. Let $\s_{1'}$ be the map swapping $1$ and $-2$, $2$ and $-1$, and fixing all other letters. The symmetric group $\symm_n$ is generated by $\{\s_i : 1 \le i \le n-1\}$. The hyperoctahedral group $\symm^B_n$ is generated by $\{\s_i : 0 \le i \le n-1\}$. The finite Coxeter group of type $D$, $\symm^D_n$, is generated by $\{\s_i : 1 \le i \le n-1\text{ or } i = 1'\}$. These involutions satisfy the relations:
\begin{align*}
(\s_i \s_j)^2 &= 1 \text{ when } i,j \in \{0,1,2,\ldots\} \text{ and } |i-j| > 1,\\
(\s_i \s_j)^3 &= 1 \text{ when } i,j \in \{1,2,3\ldots\} \text{ and } |i-j| = 1,\\
(\s_0 \s_1)^4 &= 1,\\
(\s_{1'} \s_i)^2 &= 1 \text{ when } i \in \{1,2,3, \ldots\} \setminus \{2\} \text{, and}\\
(\s_{1'} \s_2)^3 &= 1.
\end{align*}
\end{definition}

Writing group elements as ``efficient'' products of these generators is important for a variety of mathematical questions and implications.

\begin{definition}\label{defn:reduced decomposition}
Let $w$ be an element in a Coxeter group $G$ with generators $S$. If $w = s_1s_2\ldots s_{\ell(w)}$ with $s_i \in S$ and $\ell(w)$ minimal, then $s_1s_2\ldots s_{\ell(w)}$ is a \emph{reduced decomposition} for $w$ and $\ell(w)$ is the \emph{length} of $w$. The set of reduced decompositions of $w$ is denoted $R(w)$.
\end{definition}

The elements of $R(w)$ are related to each other by commutation and braid moves \cite{matsumoto, tits}, a fact that leads to many interesting questions and properties.

\begin{definition}
A Coxeter group element $w$ is \emph{fully commutative} if any two reduced decompositions for $w$ are related (only) by a sequence of commutations.
\end{definition}

For elements of the finite Coxeter group of type $A$, the reduced decompositions $R(w)$ were enumerated by Stanley in \cite{stanley}. 
The sizes of this set and important partitions of it have also been studied by others, including \cite{bergeron ceballos labbe, elnitsky, fishel milicevic patrias tenner, tenner rdpp, tenner rwm, zollinger}. Reduced decompositions can be used to endow a Coxeter group with a partial ordering, and there are two ``canonical'' ways to do this.

\begin{definition}\label{defn:strong bruhat}
For $v,w \in G$, a Coxeter group, say that $v \le w$ if there exists a reduced decomposition for $v$ that is a subword of a reduced decomposition for $w$. The resulting poset is the \emph{Bruhat order}.
\end{definition}

The \emph{subword property} says that, given a reduced decomposition of $w$, we have $v \le w$ if and only if $v$ is equal to a subword of that decomposition.

There are two versions of the other partially ordering on Coxeter groups that we consider here. Without loss of generality, we restrict our discussion to the ``right'' one.

\begin{definition}\label{defn:weak bruhat}
For $v,w \in G$, a Coxeter group, say that $v \weak w$ if there exists a reduced decomposition for $v$ that is the prefix of a reduced decomposition for $w$. The resulting poset is the \emph{(right) weak order}. For $v \weak w$, we will write $[v,w]\weakint$ for the interval between $v$ and $w$ in the weak order.
\end{definition}

Certainly $\symm_n \subset \symm^D_n \subset \symm^B_n$. A symmetric group element may be called ``unsigned,'' and any symmetric group element is also a hyperoctahedral group element. We view permutations as maps, and compose them from right to left. Thus 
$$\Big(w\s_i\Big)(j) = \begin{cases}
w(j) & \text{ if } j \not\in\{i,i+1\},\\
w(i+1) & \text{ if } j = i, \text{ and}\\
w(i) & \text{ if } j = i+1.
\end{cases}$$

In addition to writing elements of $\symm_n$, $\symm^B_n$, and $\symm^D_n$ as products of generators, we may also write them in \emph{one-line notation}, as in $w = w(1)w(2)\cdots w(n)$. Note that even in the case of signed permutations, this representation completely describes $w$.

\begin{example}\label{ex:permutations and signed permutations}
Let $x = 3214$, $y = 3\ul{2}14$, and $z = 3\ul{2}1\ul{4}$.
$$\{x,y,z\} \cap \symm_4 = \{x\} \hspace{.5in} \{x,y,z\} \cap \symm^D_4 = \{x,z\} \hspace{.5in} \{x,y,z\} \subset \symm^B_4$$
\end{example}

\subsection{Patterns and reduced decompositions}\label{ssec:old pattern results}

As part of our discussions, we will want to use the language of permutation patterns.

\begin{definition}\label{defn:patterns}
Let $p \in \symm^B_k$ and $w \in \symm^B_n$ be (possibly unsigned) permutations. The permutation $w$ \emph{contains} a \emph{$p$-pattern} if there exist indices $1 \le i_1 < \cdots < i_k \le n$ such that
\begin{itemize}
\item the string $|w(i_1)| \cdots |w(i_k)|$ is in the same relative order as $|p(1)| \cdots |p(k)| \in \symm_k$, and
\item $w(i_j) \cdot p(j) > 0$ for all $j$.
\end{itemize}
If $w$ does not contain a $p$-pattern, then $w$ \emph{avoids} $p$.
\end{definition}

\begin{example}\label{ex:patterns}
Let $p = 32\ul{1}$. The permutation $531\ul{2}4$ contains $p$, while the permutation $53124$ does not contain $p$.
\end{example}

As shown in \cite{tenner rdpp} and \cite{tenner rwm}, there is an important relationship between permutation patterns and reduced decompositions. The main results of those papers yield a key tool in some of the work below as it relates to the symmetric group. We present the implications of those earlier results as corollaries here, and refer the reader to earlier work for more details.

\begin{corollary}[{cf.~\cite{tenner rdpp, tenner rwm}}]\label{cor:what it means to avoid 321 and 3412}
For an unsigned permutation $w$, the following are equivalent:
\begin{itemize}
\item $w$ avoids the patterns $321$ and $3412$,
\item some reduced decomposition of $w$ is a product of distinct generators, and
\item every reduced decomposition of $w$ is a product of distinct generators.
\end{itemize}
\end{corollary}

The equivalence of the second and third bullet points in Corollary~\ref{cor:what it means to avoid 321 and 3412} holds for all Coxeter groups, as a result of the types of relations that may occur in these groups.

\begin{corollary}[{cf.~\cite{bjs, tenner rdpp, tenner rwm}}]\label{cor:what it means to avoid 321}
For an unsigned permutation $w$, the following are equivalent:
\begin{itemize}
\item $w$ is fully commutative,
\item $w$ avoids the pattern $321$,
\item no reduced decomposition of $w$ contains $\s_i\s_{i+1}\s_i$ as a factor, for any $i$, and
\item no reduced decomposition of $w$ contains $\s_{i+1}\s_i\s_{i+1}$ as a factor, for any $i$.
\end{itemize}
\end{corollary}

\section{Architecture of principal order ideals}\label{sec:poi lattices}

In this section, we address the structure of principal order ideals in the Bruhat order of Coxeter groups by characterizing those elements whose principal order ideals have desirable poset-theoretic features. We do this on the way to our analysis of more general intervals. 

Some of these results have appeared previously, as cited below. Here, we expand on them in several different directions. We collect and summarize these results in Table~\ref{table:poi lattices} for general Coxeter groups, and Table~\ref{table:poi lattices in Sn} for the symmetric group in terms of pattern avoidance.

\begin{table}[htbp]
{\renewcommand{\arraystretch}{2}
\begin{tabular}{m{.85in}||m{2.5in}||m{2.65in}}
& Bruhat order & Weak order\\
\hline
\hline
Lattice & Products of distinct generators \cite{ragnarsson tenner 1} & All elements of finite Coxeter groups \newline (see, for example, \cite{bjorner brenti})\\
\hline
Modular or Distributive & Products of distinct generators & Fully commutative elements (see \cite{stembridge} for Distributive)\\
\hline
Boolean & Products of distinct generators \cite{ragnarsson tenner 1} & Products of commuting generators
\end{tabular}
}
\vspace{.1in}
\caption{Characterization of Coxeter group elements whose principal order ideals have certain properties in the Bruhat and weak orders.}\label{table:poi lattices}
\end{table}
\begin{table}[htbp]
{\renewcommand{\arraystretch}{2}
\begin{tabular}{m{.85in}||m{2in}||m{2in}}
& Bruhat order on $\symm_n$ & Weak order on $\symm_n$ \\
\hline
\hline
Lattice & $321$- and $3412$-avoiding \cite{tenner patt-bru} & All permutations\\
\hline
Modular or Distributive & $321$- and $3412$-avoiding & $321$-avoiding\\
\hline
Boolean & $321$- and $3412$-avoiding \cite{tenner patt-bru} & $321$-, $231$-, and $312$-avoiding
\end{tabular}
}
\vspace{.1in}
\caption{Pattern characterization of symmetric group elements whose principal order ideals have certain properties in the Bruhat and weak orders.}\label{table:poi lattices in Sn}
\end{table}
\begin{table}[htbp]
{\renewcommand{\arraystretch}{2}
\begin{tabular}{m{.85in}||m{1.5in}||m{1.5in}}
& Bruhat order on $\symm_n$ & Weak order on $\symm_n$ \\
\hline
\hline
Lattice & $F_{2n-1}$ & $n!$\\
\hline
Modular or Distributive & $F_{2n-1}$ & $C_n$ \\
\hline
Boolean & $F_{2n-1}$ & $F_{n+1}$
\end{tabular}
}
\vspace{.1in}
\caption{Number of principal order ideals in $\symm_n$ having certain properties in the Bruhat and weak orders, where $F_i$ and $C_i$ are the $i$th Fibonacci and Catalan numbers, respectively, indexed so that $F_0 = 0$ and $F_1 = 1$.}\label{table:poi lattice enumerations in Sn}
\end{table}

The pattern-avoiding permutations described in Table~\ref{table:poi lattices in Sn} are entries P0006, P0002, and P0026, respectively, of \cite{dppa}. The enumerations given in Table~\ref{table:poi lattice enumerations in Sn} are, respectively, sequences A001519 (refined in A105306), A000108, and A000045 of \cite{oeis}. It is interesting to note that the $321$- and $3412$-avoiding permutations were enumerated independently by Fan \cite{fan} and West \cite{west}, where the former was studying products of distinct generators and the latter was studying pattern avoidance. We note that some of the Bruhat data in Tables~\ref{table:poi lattices in Sn} and~\ref{table:poi lattice enumerations in Sn} was also computed for $\symm_n^B$ and $\symm_n^D$ in \cite{tenner patt-bru}.

For the remainder of this paper, let $\poi{w}$ denote the principal order ideal of an element $w$ in the Bruhat order, and let $\weakpoi{w}$ denote the principal order ideal of $w$ in the weak order.

\subsection{Principal order ideals in the Bruhat order}

Consider the Bruhat order of a Coxeter group. As we will see, there is an intriguing amount of collapse that occurs in the hierarchy described in \eqref{eqn:hierarchy} for this poset.

We proved the first result, about boolean principal order ideals, in \cite{tenner patt-bru} for the finite Coxeter groups of types $A$, $B$, and $D$. With Ragnarsson, we expanded that result to all Coxeter groups in \cite{ragnarsson tenner 1}. As remarked in \cite{tenner patt-bru}, the ideal $\poi{w}$ is boolean in the Bruhat order if and only if it is a lattice (see Brenti's work \cite{brenti}). Despite that equivalence, we include an independent proof of the following result to give it a place in the literature, and to set the stage for future arguments.

\begin{theorem}\label{thm:lattice iff distinct letters in strong}
Let $G$ be a Coxeter group and $w \in G$. Consider $G$ as a poset under the Bruhat order. Then the principal order ideal of $w$ is a lattice if and only if some/every reduced decomposition of $w$ is a product of all distinct generators.
\end{theorem}

\begin{proof}
If some (equivalently, every) reduced decomposition of $w$ is a product of distinct generators, then it follows from \cite{tenner patt-bru} that $\poi{w}$ is a boolean poset, which itself is a lattice.

Suppose, instead, that some (equivalently, every) reduced decomposition of $w$ has a repeated letter. That is, there is a reduced decomposition $\cdots s \cdots s \cdots \in R(w)$ with $s \in S$, the generating set of $G$. In fact, because $s$ is an involution, there must be a letter $t$ appearing between the two copies of $s$ in this product, such that $t$ and $s$ do not commute:
$$\cdots s \cdots t \cdots s \cdots \in R(w).$$
Therefore, in the Bruhat order, the elements
$$s, t, st, ts \in \poi{w}.$$
are distinct. Moreover, the four elements $\{s,t,st,ts\}$ appear in the principal order ideal $\poi{w}$ with the structure shown in Figure~\ref{fig:poset examples}(a). Therefore neither the join $s \vee t$ nor the meet $st \wedge ts$ is well defined, and so $\poi{w}$ is not a lattice.
\end{proof}

By the hierarchy described in \eqref{eqn:hierarchy}, this gives the entire characterization that we seek, stated in the second column of Table~\ref{table:poi lattices}. Its translation to the language of pattern avoidance in the symmetric group, described in Table~\ref{table:poi lattices in Sn}, follows from \cite{tenner rdpp, tenner rwm}.

\subsection{Principal order ideals in the weak order}\label{sec:poi weak}

This section is similar to the last, except that we consider the weak order on a Coxeter group. In this setting, we see less collapse of the hierarchy than we saw for the Bruhat order. Consequently, we break the results into two propositions.

\begin{proposition}\label{prop:weak modular}
Let $G$ be a Coxeter group and $w \in G$. Consider $G$ as a poset under the weak order. The principal order ideal of $w$ is a modular lattice if and only if $w$ is fully commutative.
\end{proposition}

\begin{proof}
If $w$ is fully commutative, then $\weakpoi{w}$ is a distributive lattice \cite{stembridge}, so $\weakpoi{w}$ must be modular as well. Now suppose that $\weakpoi{w}$ is not fully commutative; that is, there are a reduced decompositions $x(stst\cdots)y,x(tsts\cdots)y \in R(w)$, where the factors ``$(stst\cdots)$'' and ``$(tsts\cdots)$'' each contain $m(s,t)$ letters, and $x$ and $y$ are (possibly empty) products. Then the five group elements
$$\{x,xs,xt,xst, x(stst\cdots) = x(tsts\cdots)\}$$
form a sublattice of $\weakpoi{w}$ that is isomorphic to the poset depicted in Figure~\ref{fig:poset examples}(b), and thus show that $\weakpoi{w}$ is not modular. Indeed:
$$xs \vee (xt \wedge xst) = xs \vee x = xs,$$
while
$$(xs \vee xt) \wedge xst = x(stst\cdots) \wedge xst = xst.$$
\end{proof}

\begin{proposition}\label{prop:weak boolean}
Let $G$ be a Coxeter group and $w \in G$. Consider $G$ as a poset under the weak order. The principal order ideal of $w$ is boolean if and only if $w$ is a product of commuting generators.
\end{proposition}

\begin{proof}
Boolean lattices are distributive posets, so to assume that $\weakpoi{w}$ is boolean means we can assume that $w$ is fully commutative \cite{stembridge}. Suppose that $\cdots s \cdots t \cdots \in R(w)$, with $st \neq ts$. Then because $w$ is fully commutative, it has no reduced decomposition with leftmost letter $t$, and so $t \not\weak w$. Therefore $\weakpoi{w}$ is not isomorphic to the (boolean) poset of subsets of its generators ordered by inclusion, meaning that $\weakpoi{w}$ is not boolean. This is a contradiction, and thus $w$ must be a product of commuting generators.

Conversely, if $w$ is a product of commuting generators, then for any subset $X$ of its generators, the permutation $w$ has a reduced decomposition in which the letters of $X$ form a prefix of that product. Therefore $\weakpoi{w}$ is isomorphic to the (boolean) poset of subsets of the generators of $w$, ordered by inclusion, where the subset $X$ corresponds to the element $\prod_{x \in X} x \weak w$.
\end{proof}

The pattern characterizations of these results for the symmetric group, stated in Table~\ref{table:poi lattices in Sn}, follow from \cite{tenner rdpp, tenner rwm}. The permutations described by Proposition~\ref{prop:weak boolean}, which avoid $231$, $312$, and $321$, were called ``free'' permutations in \cite{petersen tenner}.

\section{The key to boolean structures in the Bruhat order}\label{sec:boolean review}

As referenced earlier, substantial attention has been paid to boolean principal order ideals in the Bruhat order. In \cite{tenner patt-bru}, we first characterized these structures in the finite Coxeter groups of types $A$, $B$, and $D$, foreshadowing the more general statement in \cite{ragnarsson tenner 1}. Recent work of Gao and H\"anni uses Billey-Postnikov patterns to present a simpler characterization in types $B$ and $D$ \cite{gao hanni}.

As shown in Table~\ref{table:poi lattices in Sn}, boolean elements in type $A$ can be characterized by pattern avoidance. In fact, the same can be said for types $B$ and $D$, although the collection of patterns to be avoided grows. In type $A$, these patterns are $321$ and $3412$ (see Corollary~\ref{cor:what it means to avoid 321 and 3412}). The lists for types $B$ and $D$ are ten and twenty patterns long, respectively.

These so-called ``boolean'' elements in Coxeter groups have been studied in great detail, and from a variety of perspectives, in \cite{claesson kitaev ragnarsson tenner, ragnarsson tenner 1, ragnarsson tenner 2, tenner patt-bru}. Moreover, in addition to the resutling boolean properties of permutations, the patterns $321$ and $3412$ have also proven relevant in other contexts, including those mentioned in \cite{petersen tenner, tenner repetition}. Hultman and Vorwerk looked at boolean elements in the Bruhat order on involutions, which again could be characterized by pattern avoidance \cite{hultman vorwerk}.

We use this section to highlight how those results create the environment necessary for an order ideal---or, more generally, an interval---to be boolean. Suppose that $[v,w]$ is a boolean interval in the Bruhat order. This interval must be isomorphic to the boolean poset on $\ell(w)-\ell(v)$ elements. The subword property of the Bruhat order means that an interval would fail to be boolean if and only if some sort of ``collapse'' would occur when letters of a reduced decomposition of $w$ are deleted. In other words, when deleting letters to find elements between $w$ and $v$, we must be wary of running Coxeter relations, either among the letters getting deleted, or between those letters and the substring whose product yields $v$. For example, suppose that $m(s,t) = 4$. Then the product
$$stst$$
is reduced, whereas the product
$$(sts)(ts) = (tsts)s = tst(ss) = tst$$
collapses.

One might worry that there are infinitely many cases to consider. However, for any $m(s,t)$, such a collapse will require (at least) that a generator (``$s$'' in the above example) gets squared---and all generators are involutions so squaring causes a length collapse. Thus the key to understanding boolean structures in the Bruhat order is detecting when a generator can land next to a copy of itself (and how to prevent this).

\section{Hierarchical architecture of intervals}\label{sec:interval hierarchy}

We now show that the hierarchies of principal order ideals that we found in Section~\ref{sec:poi lattices} is almost a template for more general intervals in both the Bruhat and weak orders. 

\subsection{Arbitrary intervals in the Bruhat order}

Let $G$ be a Coxeter group, viewed as a poset under the Bruhat order. Recall from Section~
\ref{sec:poi lattices} that a principal order ideal in $G$ is a lattice if and only if it is boolean. Although this does not hold true for arbitrary intervals, a nearly identical statement does.

We begin with an example to show that not all lattice intervals in the Bruhat order are modular.

\begin{example}\label{ex:non-modular lattice}
The interval $[\s_k, \s_k\s_{k-1}\s_{k+1}\s_k] \subset \symm_n$ for $n \ge 4$ is a lattice, but it is not modular. Indeed, as illustrated in Figure~\ref{fig:non-modular lattice}, this interval contains a copy of Figure~\ref{fig:poset examples}(b).
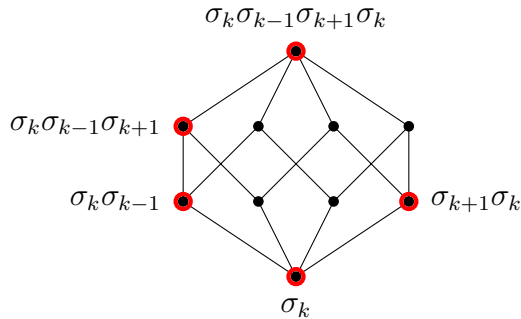
\begin{figure}[htbp]
\begin{tikzpicture}[scale=1]
\foreach \x in {(0,0),(-1.5,1),(-.5,1),(.5,1),(1.5,1),(-1.5,2),(-.5,2),(.5,2),(1.5,2),(0,3)} {\fill \x circle (2pt);}
\foreach \x in {(-1.5,1),(-.5,1),(.5,1),(1.5,1)} {\draw (0,0) -- \x;}
\foreach \x in {(-1.5,2),(-.5,2),(.5,2),(1.5,2)} {\draw (0,3) -- \x;}
\draw (-1.5,1) -- (-1.5,2) -- (-.5,1) -- (.5,2) -- (1.5,1) -- (1.5,2) -- (.5,1) -- (-.5,2) -- (-1.5,1);
\draw (0,-.15) node[below] {$\s_k$};
\draw (0,3.15) node[above] {$\s_k\s_{k-1}\s_{k+1}\s_k$};
\draw (-1.65,1) node[left] {$\s_k\s_{k-1}$};
\draw (-1.65,2) node[left] {$\s_k\s_{k-1}\s_{k+1}$};
\draw (1.65,1) node[right] {$\s_{k+1}\s_k\phantom{\s_{k-1}}$};
\foreach \x in {(0,0), (-1.5,1), (1.5,1), (-1.5,2), (0,3)} {\draw[ultra thick,red] \x circle (2.85pt);}
\end{tikzpicture}
\caption{The five elements marked in red form a copy of Figure~\ref{fig:poset examples}(b), demonstrating that this lattice is not modular.}\label{fig:non-modular lattice}
\end{figure}
\end{example}

The more interesting result, perhaps, is that if an interval in the Bruhat order is a modular lattice, then it is, in fact, a boolean lattice.

\begin{theorem}\label{thm:bruhat intervals modular=boolean}
Let $G$ be a Coxeter group, viewed as a poset under the Bruhat order. An interval $[v,w]$ in $G$ is boolean if and only if it is a modular lattice.
\end{theorem}

\begin{proof}
If $[v,w]$ is boolean, then it is also a modular lattice, by definition.

Now suppose that $[v,w]$ is a modular lattice. If the interval has rank $2$ (that is, $\ell(w) - \ell(v) = 2$), then this interval is isomorphic to the boolean poset on two elements. Jantzen showed that a rank $3$ interval in $G$ must have one of the following forms \cite{jantzen}:
\begin{itemize}
\item it is isomorphic to the poset in Figure~\ref{fig:2-crown} (in which case it is not a lattice),
\item it is isomorphic to the boolean poset on three elements, or
\item it is isomorphic to the poset in Figure~\ref{fig:non-modular lattice} (in which case it is not a modular lattice).
\end{itemize}
Thus, if $[v,w]$ is a modular lattice, then every rank $3$ interval that it contains must be a boolean lattice. Now consider any interval $[a,b]$ in $[v,w]$, such that $[a,b]$ has rank at least $4$. Bj\"orner showed in \cite{bjorner} that the open interval $(a,b)$ is isomorphic to the face poset of a regular CW-decomposition of a sphere of dimension $\ell(b) - \ell(a) - 2$. Because $\ell(b) - \ell(a) - 2 \ge 2$, this face poset is connected. Thus, it follows from a result of Stanley (see Grabiner's discussion and proof \cite[Lemma 8]{grabiner}), that $[v,w]$ itself is boolean.
\end{proof}

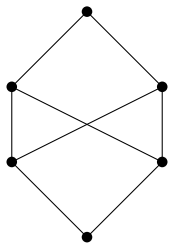
\begin{figure}[htbp]
\begin{tikzpicture}[scale=1]
\coordinate (xy) at (0,0);
\coordinate (xsy) at (-1,1);
\coordinate (xty) at (1,1);
\coordinate (xsty) at (-1,2);
\coordinate (xtsy) at (1,2);
\coordinate (xstsy) at (0,3);
\foreach \x in {xy,xsy,xty,xsty,xtsy,xstsy} {\fill (\x) circle (2pt);}
\foreach \x in {xsy,xty} {\draw (\x) -- (xy); \draw (\x) -- (xsty); \draw (\x) -- (xtsy);}
\foreach \x in {xsty,xtsy} {\draw (\x) -- (xstsy);}
\end{tikzpicture}
\caption{One of the three possible forms of a rank $3$ interval in a Coxeter group.}\label{fig:2-crown}
\end{figure}

Therefore, in the Bruhat order, the hierarchy of intervals is
$$\Big\{\text{boolean}\Big\} = \Big\{\text{distributive}\Big\} = \Big\{\text{modular}\Big\} \subsetneq \Big\{\text{lattice}\Big\}.$$

\subsection{Arbitrary intervals in the weak order}

In fact, a classical result about the weak order implies that the goals of this section were already addressed in Section~\ref{sec:poi weak}.

\begin{proposition}[{\cite[Proposition 3.1.6]{bjorner brenti}}]\label{prop:weak intervals are pois}
If $v \weak w$, then $[v,w]_{\text{wk}} \cong \weakpoi{v^{-1}w}$.
\end{proposition}

We can now use Propositions~\ref{prop:weak modular} and~\ref{prop:weak boolean} to describe the well-behaved intervals in the weak order.

\begin{corollary}\label{cor:weak interval characterization}
Let $G$ be a Coxeter group, viewed as a poset under the weak order, and consider $v,w \in G$ with $v \weak w$. The interval $[v,w]\weakint$ is
\begin{itemize}
\item always a lattice, 
\item modular if and only if $v^{-1}w$ is fully commutative, 
\item distributive if and only if $v^{-1}w$ is fully commutative, and
\item boolean if and only if $v^{-1}w$ is the product of commuting generators.
\end{itemize}
\end{corollary}

\section{Intervals above atoms in the symmetric group}\label{sec:intervals over atoms in Sn}

The results of Section~\ref{sec:interval hierarchy} are encouraging for understanding the structure of arbitrary intervals in either of these partial orders. However, it is not easy to create an analogue of Table~\ref{table:poi lattices in Sn} for arbitrary intervals. Indeed, despite the results above, there are a number of ways in which the characterization question for intervals is more complex than the question for principal order ideals---particularly for the Bruhat order. Suppose that $v \le w$ in the Bruhat order of a Coxeter group, and fix a reduced decomposition $s_1 \cdots s_{\ell} \in R(w)$. One added complexity is that there might be more than one reduced decomposition of $v$ appearing as a subword of $s_1 \cdots s_{\ell(w)}$, and possibly even more than one occurrence of the same reduced decomposition of $v$. Another intricacy is that while we needed all generators to be distinct for a principal order ideal, this might not be the case among the generators in $s_1 \cdots s_{\ell(w)}$ that are outside of $s_{i_1}\cdots s_{i_{\ell(v)}} \in R(v)$, because identical generators might be prevented from interacting by some $s_{i_j}$. Finally, and perhaps most challenging of all, it may be impossible to find a reduced decomposition of $v$ appearing as a factor inside of a reduced decomposition of $w$. That is, it may be necessary for reduced decompositions of $w$ to ``interrupt'' reduced decompositions of $v$, as is the case for the interval $[2143,2341] \subset \symm_4$, since $R(2143) = \{\s_1\s_3,\s_3\s_1\}$ and $R(2341) = \{\s_1\s_2\s_3\}$.

In this section, we give characterizations of boolean (and distributive, and modular) intervals and of lattice intervals in the Bruhat order, for which the bottom element has rank $1$. This will build on the ``rank $0$'' case of principal order ideals described in Section~\ref{sec:poi lattices}. Already, this is complicated to state, and one can see the difficulties described above coming into play. We will do similarly for the weak order, which has a more elegant answer, as suggested by Corollary~\ref{cor:weak interval characterization}. In both settings, we will enumerate these intervals, with rather pleasing results.

\subsection{Intervals above atoms in the Bruhat order}

By Theorem~\ref{thm:bruhat intervals modular=boolean}, to understand the well-behaved lattices of the form $[\s_k,w]$ in the symmetric group, it is almost enough to understand when such an interval is a boolean poset.

\begin{theorem}\label{thm:boolean interval at rank 1}
Consider $\symm_n$  as a poset under the Bruhat order. An interval $[\s_k,w]$ is boolean (equivalently, a modular lattice, or a distributive lattice) if and only if
$x \s_k y \in R(w)$ such that
\begin{itemize}
\item $x$ and $y$ are each (possibly empty) products of distinct generators,
\item the only generators that may appear in both $x$ and $y$ are $\s_{k\pm1}$, and
\item $\s_k$ does not appear in $x$ or $y$.
\end{itemize}
\end{theorem}

\begin{proof}
First observe that if there is such an $x \s_k y \in R(w)$, then certainly $[\s_k,w]$ is boolean.

Now suppose that $[\s_k,w]$ is boolean. We first note that we can find a reduced decomposition for $w$ in which only one copy of $\s_k$ appears. Suppose $a\s_kb\s_kc \in R(w)$, with $\s_k$ not appearing in the product $b$. Let $N$ be the number of copies of $\s_k$ that appear in this decomposition. If $b$ were to contain both $\s_{k-1}$ and $\s_{k+1}$, then the interval $[\s_k,w]$ would contain the interval depicted in Figure~\ref{fig:non-modular lattice}, which is not boolean.  Thus, without loss of generality, the product $b$ contains $\s_{k-1}$ and not $\s_{k+1}$, and so we can do commutation moves on this product to find a reduced decomposition $a' \s_k \s_{k-1} \s_k c' \in R(w)$. This product still has $N$ copies of $\s_k$, and a braid move produces the reduced decomposition $a' \s_{k-1} \s_k \s_{k-1} c' \in R(w)$ having $N-1$ copies of $\s_k$. In this way, we can systematically produce reduced decompositions of $w$ with fewer copies of $\s_k$ until we find $x \s_k y \in R(w)$ with exactly one copy of $\s_k$.

This particular decomposition is helpful because it means that in the interval $[\s_k,w]$, we know exactly what letters must get deleted from the reduced decomposition $x \s_k y \in R(w)$, simplifying many of the concerns discussed at the opening of this section. Indeed, all of the generators in $x$ and $y$ must be deleted. If there are any repeated generators that could land next to each other by deletions of these letters, then the interval will fail to be boolean. This scenario is avoided if and only if $x$ and $y$ are each, themselves, products of distinct generators, and if the only generators they have in common cannot commute past $\s_k$.
\end{proof}

The three conditions in the statement of Theorem~\ref{thm:boolean interval at rank 1} have some interesting implications for the reduced decompositions of such a $w$. Indeed, suppose that $x\s_ky \in R(w)$ meets the three conditions of the theorem, and contains $a$ copies of $\s_{k+1}$. First, due to those conditions, we must have $a \in \{0,1,2\}$. Certainly if $a = 0$, then $\s_{k+1}$ does not appear in any reduced decompositions of $w$. If, on the other hand, $a = 1$, then there is one copy of $\s_k$ (by the third requirement in the statement of the theorem), one copy of $\s_{k+1}$, and at most one copy of $\s_{k+2}$ (by the second requirement in the statement of the theorem). Thus, the only Coxeter moves that can involve the one copy of $\s_{k+1}$ are commutations (not braids), and so in fact there is exactly one copy of $\s_{k+1}$ in all reduced decompositions of $R(w)$. Finally, suppose that $a = 2$. Then the two copies of $\s_{k+1}$ are separated by the one copy of $\s_k$ in the decomposition. Although there may be other reduced decompositions of $w$ containing just one copy of $\s_{k+1}$, such a decomposition would necessarily contain more than one copy of $\s_k$, and thus would not satisfy the conditions of Theorem~\ref{thm:boolean interval at rank 1}.

Notice that Theorem~\ref{thm:boolean interval at rank 1} relies on the fact that the exponents $m(s,t)$ in $\symm_n$ are all $2$ or $3$. In contrast, boolean intervals in the hyperoctahedral group need not have the property that one can find a reduced decomposition of the top element of the interval containing only one copy of the bottom element of the interval.

\begin{example}
The interval $[\ul{1}2,\ul{2}\ul{1}] \subset \symm_2^B$ is boolean, even though $\ul{1}2 = \s_0$ appears twice in the only reduced decomposition $\s_0\s_1\s_0$ of $\ul{2}\ul{1}$.
\end{example}

Recall that there are $F_{2n-1}$ boolean principal order ideals in the Bruhat order on $\symm_n$. The number of boolean intervals whose minimum elements are atoms is notably larger as we see in the next theorem. In fact, only for $k \in [1,n-1]$ are there fewer boleean intervals over a particular atom $\s_k$ than there are boolean principal order ideals. For comparison, see Table~\ref{table:boolean interval at rank 1 count} below.

\begin{theorem}\label{thm:boolean interval at rank 1 count}
Fix $n > 2$. Consider $\symm_n$ as a poset under the Bruhat order, and fix $k \in [1,n-1]$. The number of boolean intervals (equivalently, modular lattices, or distributive lattices) of the form $[\s_k,w]$ is
\begin{equation}\label{eqn:boolean interval count}
t(n,k) := \begin{cases}
4F_{2n-4} & \text{ if } k \in \{1,n-1\}, \text{ and}\\
16F_{2k-2}F_{2(n-k)-2} & \text { if } k \in [2,n-2].
\end{cases}
\end{equation}
where $F_i$ is the $i$th Fibonacci number.
\end{theorem}

\begin{proof}
Set $T(n,k) := \{w \in \symm_n : [\s_k,w] \text{ is boolean}\}$, and thus $t(n,k) = |T(n,k)|$. Note that, by symmetry, $t(n,k) = t(n,n-k)$. We begin by computing $t(n,1) = t(n,n-1)$, and then use this to compute $t(n,k)$ for $k \in [2,n-2]$.

Consider the implications of Theorem~\ref{thm:boolean interval at rank 1} when $k = 1$: the permutation $w$ must have a reduced decomposition $x\s_1y$ such that $x$ and $y$ are products of distinct generators, the only generator they might have in common is $\s_2$, and neither contains $\s_1$. The set $T(n,1)$ can be partitioned into three parts, based on how many copies of $\s_2$ appear in $x \sigma_1y$:
\begin{itemize}
\item $T_0(n,1) := \{w \in T(n,1) : w = x\s_1y \text{ contains no } \s_2\}$,
\item $T_1(n,1) := \{w \in T(n,1) : w = x\s_1y \text{ contains one } \s_2\}$, and
\item $T_2(n,1) := \{w \in T(n,1) : w = x\s_1y \text{ contains two } \s_2\text{s}\}$.
\end{itemize}
As discussed after the proof of Theorem~\ref{thm:boolean interval at rank 1}, this decomposition is well-defined, and thus $t(n,1)$ is the sum of the sizes of these three sets.

The first set, $T_0(n,1)$ can be described as taking $xy$, a product of distinct generators from $\{\s_3,\ldots,\s_{n-1}\}$, and multiplying it by $\s_1$. Because $\s_1$ commutes with all elements of $\{\s_3,\ldots,\s_{n-1}\}$, the placement of $\s_1$ does not affect the product. As shown in Table~\ref{table:poi lattice enumerations in Sn}, there are $F_{2(n-2)-1}$ such products $xy$.

For $T_1(n,1)$, take $xy$, a product of distinct generators from $\{\s_2,\ldots,\s_{n-1}\}$ that must include $\s_2$, and multiply it on the left or the right (these are distinct permutations), by $\s_1$. There are $F_{2(n-1)-1}$ products of distinct generators from $\{\s_2,\ldots,\s_{n-1}\}$, and $F_{2(n-2)-1}$ of those are actually products of distinct generators from $\{\s_3,\ldots,\s_{n-1}\}$. Thus there are $F_{2(n-1)-1} - F_{2(n-2)-1}$ products of distinct generators from $\{\s_2,\ldots,\s_{n-1}\}$ that must include $\s_2$, so
$$|T_1(n,1)| = 2\left(F_{2(n-1)-1} - F_{2(n-2)-1}\right) = 2F_{2n-4}.$$

Finally, the set $T_2(n,1)$ can be described easily in terms of whether or not it involves $\s_3$. If there is no $\s_3$, then take any product of distinct generators from $\{\s_4,\ldots,\s_{n-1}\}$ and insert $\s_2\s_1\s_2$ anywhere among the generators in that product. Because $\s_i$ and $\s_j$ commute for $i < 3 < j$, precise positioning does not affect the overall product. This yields $F_{2(n-3)-1}$ elements. On the other hand, if there is a copy of $\s_3$, then we can start with a product of distinct generators from $\{\s_3,\ldots,\s_{n-1}\}$ that must include $\s_3$, and insert $\s_2\s_1\s_2$ in any of three possible ways around this $\s_3$:
$$\s_2\s_1\s_2\s_3 \ , \  \s_2\s_1\s_3 \s_2\ , \ \s_3\s_2\s_1\s_2.$$
Note that the relative positions of all other generators will not affect the overall product. Thus
\begin{align*}
|T_2(n,1)| &= F_{2(n-3)-1} + 3\left(F_{2(n-2)-1} - F_{2(n-3)-1}\right)\\
&= F_{2n-7} + 3F_{2n-6}\\
&= F_{2n-5} + 2F_{2n-6}\\
&= F_{2n-4} + F_{2n-6}.
\end{align*}

Therefore
$$t(n,1) = F_{2n-5} + 2F_{2n-4} + F_{2n-4} + F_{2n-6} = 4F_{2n-4}.$$

For $k \in [2,n-2]$, we can form elements of $T(n,k)$ from elements $u \in T(n-k+1,1)$ and $v \in T(k+1,k)$. Then take the reduced decompositions $u = x_u\s_1y_u$ and $v = x_v\s_{k-1}y_v$ described by Theorem~\ref{thm:boolean interval at rank 1}. Let $u' = x_u' \s_k y'_u$ be obtained from $x_u\s_1y_u$ by ``shifting'' all generators according to $\s_i \mapsto \s_{i+k-1}$. Then
$$(x_u'x_v)\s_k(y_u'y_v) \in T(n,k).$$
Moreover, because $\s_i$ and $\s_j$ commute for $i < k < j$, the product $(x_u'x_v)\s_k(y_u'y_v)$ is equal to the product formed by any other interweaving of the letters to the left of $\s_k$ and similarly of those to its right. Each element of $T(n,k)$ can be formed in this way, and given an element of $T(n,k)$ it is straightforward to compute its corresponding $u \in T(n-k+1,1)$ and $v \in T(k+1,k)$. Therefore
\begin{align*}
t(n,k) &= t(n-k+1,1) \cdot t(k+1,k)\\
&= t(n-k+1,1) \cdot t(k+1,1)\\
&= 4F_{2(n-k+1)-4} \cdot 4F_{2(k+1)-4},
\end{align*}
as desired.
\end{proof}

The reader can confirm the $n \in [3,4]$ cases of this result by looking ahead to Figure~\ref{fig:boolean over support bruhat}. The data computed in Theorem~\ref{thm:boolean interval at rank 1 count} is shown in Table~\ref{table:boolean interval at rank 1 count} for small values of $n$. The row sums of the table are $8$ times the entries in A054444 of \cite{oeis}. The triangle described by $k \in [2,n-2]$ is $16$ times the entries of A141678 of \cite{oeis}.

\begin{table}[htbp]
{\renewcommand{\arraystretch}{2}$\begin{array}{c||c|c|c|c|c|c|c|c||c}
& k = 1 & 2 & 3 & 4 & 5 & 6 & 7 & 8 & \sum_k\\
\hline
\hline
n = 3 & 4 & 4 & & & & & & & 8\\
\hline
4 & 12 & 16 & 12 & & & & & & 40\\
\hline
5 & 32 & 48 & 48 & 32 & & & & & 160\\
\hline
6 & 84 & 128 & 144 & 128 & 84 & & & & 568\\
\hline
7 & 220 & 336 & 384 & 384 & 336 & 220 & & & 1880\\
\hline
8 & 576 & 880 & 1008 & 1024 & 1008 & 880 & 576 & & 5952\\
\hline
9 & 1508 & 2304 & 2640 & 2688 & 2688 & 2640 & 2304 & 1508 & 18280
\end{array}$}
\vspace{.1in}
\caption{Number of boolean intervals (equivalently, modular lattices, or distributive lattices) $[\s_k,w]$ in the Bruhat order on $\symm_n$, for $n \in [3,9]$, and the total number of boolean intervals  (equivalently, modular lattices, or distributive lattices) over all atoms.}\label{table:boolean interval at rank 1 count}
\end{table}

We close this section by considering the larger class of lattice intervals above atoms in the Bruhat order. Example~\ref{ex:non-modular lattice} foreshadows the characterization quite well. To set the stage, recall Figure~\ref{fig:2-crown}, which is not a lattice.

\begin{lemma}\label{lem:avoiding a 2-crown}
Consider $\symm_n$ as a poset under the Bruhat order, and an interval $[\s_k, w]$. If there exists a reduced decomposition of $w$ in which a letter is repeated on some side of (any appearance of) $\s_k$, then $[\s_k, w]$ contains a subposet of the form depicted in Figure~\ref{fig:2-crown}, and thus $[\s_k,w]$ is not a lattice.
\end{lemma}

\begin{proof}
Suppose that there is such a reduced decomposition, and that $\s_h$ is the repeated letter. To be reduced, we must have, without loss of generality,
$$a \s_h b \s_{h+1} c \s_h d \s_k e \in R(w),$$
where $a$, $b$, $c$, $d$, and $e$ are each (possibly empty) products of generators. Then the interval $[d\s_k, \s_h\s_{h+1}\s_hd\s_k] \subseteq [\s_k,w]$ has the form depicted in Figure~\ref{fig:2-crown}.
\end{proof}

\begin{theorem}\label{thm:lattice interval at rank 1}
Consider $\symm_n$ as a poset under the Bruhat order. An interval $[\s_k,w]$ is a lattice if and only if $w$ either has a reduced word as described by Theorem~\ref{thm:boolean interval at rank 1} (in which case $[\s_k,w]$ is boolean), or it has a reduced word of the form $x(\s_k\s_{k-1}\s_{k+1}\s_k)y \in R(w)$ such that
\begin{itemize}
\item $x$ and $y$ are each (possibly empty) products of distinct generators,
\item no generator appears in both $x$ and $y$, and
\item no element of $\{\s_k,\s_{k\pm1}\}$ appears in either $x$ or $y$.
\end{itemize}
\end{theorem}

\begin{proof}
Following Theorem~\ref{thm:boolean interval at rank 1}, it suffices to consider when $[\s_k,w]$ is a non-boolean lattice. So $u \s_k v \in R(w)$ where $u$ and $v$ are each (possibly empty) products of generators. Recall Lemma~\ref{lem:avoiding a 2-crown}. Thus $u$ and $v$ are each (possibly empty) products of distinct generators. Moreover, suppose that $u$ and $v$ each contain $\s_h$ such that $|h - k| > 1$. To be reduced, we must have $\s_{h+1}$ appearing in $u$, without loss of generality. The generators $\s_h$ and $\s_k$ commute. Therefore $\s_h\s_{h+1}\s_k\s_h = \s_h\s_{h+1}\s_h\s_k$, and so $[\s_k,s_h\s_{h+1}\s_k\s_h] \subseteq [\s_k,w]$ has the form depicted in Figure~\ref{fig:2-crown}, contradicting the assumption that $[\s_k,w]$ is a lattice. 

Therefore, using these facts and Lemma~\ref{lem:avoiding a 2-crown}, we can write $u \s_k v$ in the form
$$a \s_k b \s_k c,$$
where the words $a$, $b$, and $c$ are each a product of distinct generators, the generator $\s_k$ does not appear in any of them, the only generators that may appear in both $a$ and the concatenation $bc$ are $\s_{k\pm1}$, and similarly for $ab$ and $c$. Moreover, because we are assuming that $[\s_k,w]$ is not boolean, we can assume that $b$ contains both $\s_{k-1}$ and $\s_{k+1}$. Then, in fact, we can apply commutation relations to this product to obtain $x(\s_k\s_{k-1}\s_{k+1}\s_k)y \in R(w)$. If any of the three requirements itemized in the theorem statement fails to hold, then the interval $[\s_k,w]$ will contain a copy of Figure~\ref{fig:2-crown} by Lemma~\ref{lem:avoiding a 2-crown}, and hence the interval will not be a lattice.

Now suppose that $w$ has a reduced decomposition as described in the latter statement of the theorem. Then, in fact,
\begin{equation}\label{eqn:lattice interval as product}
[\s_k,w] \cong [\s_k,\s_k\s_{k-1}\s_{k+1}\s_k] \times [\s_k,x'\s_ky'],
\end{equation}
where $x'$ is obtained from $x$ (and $y'$ from $y$) by replacing each generator according to the rule:
\begin{equation}\label{eqn:shifting letters}
\begin{cases}
\s_i \mapsto \s_{i+1} & \text{if } i < k-1, \text{ and}\\
\s_i \mapsto \s_{i-1} & \text{if } i > k+1.
\end{cases}
\end{equation}
To be precise, an element of $[\s_k,w]$ can be written as $u q v$, where $u$ is a subword of $x$, $q$ is a subword of $\s_k \s_{k-1}\s_{k+1}\s_k$ that contains at least one copy of $\s_k$, and $v$ is a subword of $y$. This corresponds to $(q, u'\s_kv') \in [\s_k,\s_k\s_{k-1}\s_{k+1}\s_k] \times [\s_k,x'\s_ky']$, where $u'$ and $v'$ are constructed via \eqref{eqn:shifting letters}. A covering relation can be written as $u_1 q_1 v_1 \lessdot u_2 q_2 v_2$, where a single letter has been deleted: $\epsilon_1$ is obtained by deleting a letter of $\epsilon_2$, for some choice of $\epsilon \in {u,q,v}$, requiring that $q_1$ contains at least one copy of $\s_k$. The elements in this covering relation correspond to $(q_i, u_i' \s_k v_i')$ in the product poset, respectively, and we will necessarily have $(q_1, u_1' \s_k v_1') \lessdot (q_2, u_2'\s_k v_2')$. The reverse implication holds as well.

The poset $[\s_k,\s_k\s_{k-1}\s_{k+1}\s_k]$ is the lattice depicted in Figure~\ref{fig:non-modular lattice} and the poset $[\s_k,x'\s_ky']$ is a boolean lattice by Theorem~\ref{thm:boolean interval at rank 1}, because of the requirements on $x$ and $y$ stated above. The direct product of two lattices is itself a lattice, completing the proof.
\end{proof}

The partitioning used in Theorem~\ref{thm:lattice interval at rank 1}, separating boolean and non-boolean lattices, and the decomposition given in Equation~\eqref{eqn:lattice interval as product} can be used to enumerate lattice intervals of the form $[\s_k,w]$ in $\symm_n$. When $n \le 3$, there are no non-boolean lattice intervals, and so that enumeration is covered by Theorem~\ref{thm:boolean interval at rank 1 count}.

\begin{corollary}\label{cor:lattice interval at rank 1 count}
Fix $n > 3$. Consider $\symm_n$ as a poset under the Bruhat order, and fix $k \in [1,n-1]$. The number of lattice intervals of the form $[\s_k,w]$ is
$$\begin{cases}
4F_{2n-4} & \text{ if } k \in \{1,n-1\}, \text{ and}\\
16F_{2k-2}F_{2(n-k)-2} + F_{2n-5} - F_{2k-3}F_{2(n-k)-3} & \text { if } k \in [2,n-2],
\end{cases}$$
where $F_i$ is the $i$th Fibonacci number.
\end{corollary}

\begin{proof}
It remains to enumerate the non-boolean lattices of the form $[\s_k,w]$, which can only occur when $k \in [2,n-2]$. As described in Equation~\eqref{eqn:lattice interval as product}, we must count products of distinct generators taken from $\{\s_2,\s_3,\ldots,\s_{n-2}\}$ in which $\s_k$ appears. We do this by counting the complement. The total number of products of distinct generators taken from $\{\s_2,\s_3,\ldots,\s_{n-2}\}$ is $F_{2(n-2)-1}$. Those that do not use $\s_k$ are a combination of a product of distinct generators taken from $\{\s_2,\ldots,\s_{k-1}\}$, and a product of distinct generators taken from $\{\s_{k+1},\ldots,\s_{n-2}\}$. Note that a product of the first kind commutes with a product of the second kind. There are $F_{2(k-1)-1}$ of the former products and there are $F_{2(n-1-k)-1}$ of the latter products. Thus there are
$$F_{2n-5} - F_{2k-3}F_{2(n-k)-3}$$
non-boolean lattices of the form $[\s_k,w]$. Combining this with Theorem~\ref{thm:boolean interval at rank 1 count} completes the proof.
\end{proof}

\subsection{Intervals above atoms in the weak order}

We close this section by considering the implications of Corollary~\ref{cor:weak interval characterization} on $\symm_n$. That theorem gives a fairly complete description of the desired interval types with $\s_k$, a generator of $\symm_n$. Thus we devote this section to an enumeration of those intervals. It will be interesting to observe that only the enumeration of boolean intervals depends on the value of the index $k$.

\begin{proposition}\label{prop:counting lattices above atoms in the weak order}
For fixed $k$, the number of lattices $[\s_k,w]\weakint$ in the weak order on $\symm_n$ is $n!/2$.
\end{proposition}

\begin{proof}
By Corollary~\ref{cor:weak interval characterization}, it suffices to count the elements $w$ for which $\s_k \weak w$. There is a simple bijection between such $w$ and the elements that are not above $\s_k$ in this poset: $w \longleftrightarrow \s_kw$. Thus exactly half of the elements of $\symm_n$ are greater than or equal to $\s_k$.
\end{proof}

This is sequence A001710 of \cite{oeis}.

\begin{corollary}
The total number of lattices of the form $[\s_k,w]\weakint$ in the weak order on $\symm_n$ is $(n-1)n!/2$.
\end{corollary}

We skip ahead to boolean lattices, for a moment, which have a similarly straightforward enumeration.

\begin{proposition}\label{prop:counting boolean lattices above atoms in the weak order}
For fixed $k$, the number of boolean intervals $[\s_k,w]\weakint$ in the weak order on $\symm_n$ is $F_{k+1} F_{n-k+1}$, where $F_i$ is the $i$th Fibonacci number.
\end{proposition}

\begin{proof}
By Corollary~\ref{cor:weak interval characterization}, we must count $w$ for which $\s_k \weak w$ and $\s_kw$ is free. Thus we can write
\begin{equation}\label{eqn:weak boolean interval over atom decomposition}
w = \s_k \s_{i_1}\s_{i_2}\cdots \s_{i_{\ell(w)-1}},
\end{equation}
and the generators $\{\s_{i_1}, \s_{i_2}, \ldots, \s_{i_{\ell(w)-1}}\}$ must all commute with each other. Moreover, because Equation~\eqref{eqn:weak boolean interval over atom decomposition} is a reduced decomposition of $w$, we must have $k \neq i_j$, for all $j$. Thus $\{\s_{i_1}, \s_{i_2}, \ldots, \s_{i_{\ell(w)-1}}\}$ is the union of two sets: one a subset of commuting generators chosen from $\{\s_1,\ldots,\s_{k-1}\}$, and the other a subset of commuting generators chosen from $\{\s_{k+1},\ldots,\s_{n-1}\}$. As indicated in Table~\ref{table:poi lattice enumerations in Sn}, there are $F_{k+1}$ of the former and $F_{n-k+1}$ of the latter, completing the proof.
\end{proof}

This is sequence A106408 of \cite{oeis}.

\begin{corollary}
The total number of boolean intervals of the form $[\s_k,w]\weakint$ in the weak order on $\symm_n$ is
\begin{equation}\label{eqn:weak boolean interval over atom total}
\sum_{k=1}^{n-1} F_{k+1} F_{n-k+1} = \frac{(n+1)F_{n+3}+ (n-7)F_{n+1}}{5}.
\end{equation}
\end{corollary}

The sequence described by Equation~\eqref{eqn:weak boolean interval over atom total} is A004798 of \cite{oeis}.

It remains, now, to count distributive/modular intervals of the form $[\s_k,w]\weakint$ in the weak order.

\begin{theorem}\label{thm:counting dist/mod lattices above atoms in the weak order}
For fixed $k$, the number of distributive (equivalently, modular) intervals $[\s_k,w]\weakint$ in the weak order on $\symm_n$ is $C_n - C_{n-1}$, where $C_i$ is the $i$th Catalan number.
\end{theorem}

\begin{proof}
By Corollary~\ref{cor:weak interval characterization}, we must count $w$ for which $\s_k \weak w$ and $\s_kw$ is fully commutative. By previous results (see, for example, \cite{bjs, tenner rdpp}), this means that
\begin{equation}\label{eqn:weak dist/mod interval over atom decomposition}
w = \s_k \s_{i_1}\s_{i_2}\cdots \s_{i_{\ell(w)-1}} \in R(w),
\end{equation}
and no element of $R(\s_k w)$ contains a braid move $\s_i \s_{i\pm 1} \s_i$ as a factor.

Note, first, that the specific value of $k$ does not affect the enumeration suggested in the statement of the proposition. Indeed, we can ``cycle'' the indices of the product in Equation~\eqref{eqn:weak dist/mod interval over atom decomposition}, adding a fixed value, modulo $n-1$ to each, giving a bijection between distributive/modular intervals $[\s_k,w]\weakint$ and distributive/modular intervals $[\s_{k'},w']\weakint$. The permutation $\s_k\s_{i_2}\cdots \s_{i_{\ell}}$ would be mapped to $\s_{k'} \s_{j_2} \cdots \s_{j_{\ell}}$, where $j_h = i_h + k' - k$ modulo $n-1$, and the former avoids $\s_t\s_{t\pm1}\s_t$ factors, for all $ \in [1,n-2]$, if and only if the latter does as well. A similar idea was employed in \cite{reiner}, and later in \cite{tenner expb, tenner comm}. Thus, without loss of generality, we can assume that $k = 1$.

Consider the collection $FC_n$ of all fully commutative permutations in $\symm_n$. Define
\begin{align*}
A &:= \{v \in FC_n : \s_1 \weak v\} \text{ and}\\
B &:= \{v \in FC_n : \s_1 \not\weak v\},
\end{align*}
and note that $\{A,B\}$ is a partition of $FC_n$.

Elements in $A$ have reduced decompositions of the desired form depicted in Equation~\eqref{eqn:weak dist/mod interval over atom decomposition}, although they do not describe all such permutations.

Now consider the set $B$. For each $v \in B$, there are two possibilities for the (longer) permutation $\s_1v$:
\begin{itemize}
\item $\s_1v$ is fully commutative (i.e., $\s_1v\in FC_n$), or
\item $\s_1v$ is not fully commutative (i.e., $\s_1v \not\in FC_n$).
\end{itemize}
Let $B_1$ be the set of permutations in the first category, and $B_2$ those in the second. As before, $\{B_1,B_2\}$ is a partition $B$.

The set $B_2$ describes exactly the set of permutations alluded to in the previous paragraph: the permutations $w \ge_{\text{wk}} \s_1$ which are not, themselves, fully commutative, but for which (the shorter permutation) $\s_1w$ is fully commutative; i.e., for which $[\s_1,w]\weakint$ is distributive/modular, while $\weakpoi{w}$ is not.

Thus we must compute
$$|A| + |B_2| = |A| + |B| - |B_1| = |FC_n| - |B_1| = C_n - |B_1|,$$
where $C_n$ is the $n$th Catalan number.

It remains only to calculate $|B_1|$; i.e., to enumerate the $v \in B$ for which $v$ and $\s_1v$ are fully commutative. If a reduced decomposition of $v$ were to include the generator $\s_1$, then, because $\s_1 \not\weak v$, one could do commutation moves in $\s_1v$ to obtain a factor
$$\cdots (\s_1\s_2\s_1) \cdots.$$
But then $\s_1v$ is not fully commutative, which is a contradiction. Therefore $v$ is a fully commutative product of the generators $\{\s_2,\ldots,\s_{n-1}\}$. Note that, certainly, for any such product, the permutation $\s_1v$ must also be fully commutative. There are $C_{n-1}$ such products, and so $|B_1| = C_{n-1}$, completing the proof.
\end{proof}

This is sequence A000245 of \cite{oeis}.

\begin{corollary}
The total number of distributive (equivalently, modular) intervals of the form $[\s_k,w]\weakint$ in the weak order on $\symm_n$ is $(n-1)(C_n - C_{n-1})$.
\end{corollary}

We collect the enumerations of this section in Table~\ref{table:atom lattice enumerations in Sn}, reminiscent of Table~\ref{table:poi lattice enumerations in Sn} in Section~\ref{sec:poi lattices}.

\begin{table}[htbp]
{\renewcommand{\arraystretch}{2}
\begin{tabular}{m{.85in}||m{3.35in}||m{1.4in}}
& Bruhat order on $\symm_n$ & Weak order on $\symm_n$ \\
\hline
\hline
\raisebox{0in}[.5in][.3in]{Lattice} &
\raisebox{0in}[.5in][.4in]{$\begin{cases}
4F_{2n-4} & \text{ if } k \in \{1,n-1\}\\
16F_{2k-2}F_{2(n-k)-2} + F_{2n-5} & \\
\hspace{.25in} - F_{2k-3}F_{2(n-k)-3} & \text { if } k \in [2,n-2]
\end{cases}$} & 
\raisebox{0in}[.5in][.3in]{$n!/2$}\\
\hline
\raisebox{0in}[.3in][0in]{Modular or} Distributive & 
\raisebox{0in}[.3in][.3in]{$\begin{cases}
4F_{2n-4} & \text{ if } k \in \{1,n-1\}\\
16F_{2k-2}F_{2(n-k)-2} \phantom{ +\ \, F_{2n-5}} & \text { if } k \in [2,n-2]
\end{cases}$} &
\raisebox{0in}[.3in][.3in]{$C_n-C_{n-1}$} \\
\hline
\raisebox{0in}[.4in][.3in]{Boolean} & 
\raisebox{0in}[.4in][.3in]{$\begin{cases}
4F_{2n-4} & \text{ if } k \in \{1,n-1\}\\
16F_{2k-2}F_{2(n-k)-2} \phantom{ +\ \, F_{2n-5}} & \text { if } k \in [2,n-2]
\end{cases}$} & 
\raisebox{0in}[.4in][.3in]{$F_{k+1}F_{n-k+1}$}
\end{tabular}
}
\vspace{.1in}
\caption{Number of intervals in $\symm_n$ with minimum element $\s_k$, having certain properties in the Bruhat and weak orders.}\label{table:atom lattice enumerations in Sn}
\end{table}

\section{Further directions for research}\label{sec:open questions}

In addition to dreams of elegant and general versions of the results of Section~\ref{sec:intervals over atoms in Sn}, there are many directions for further work on this topic. We present two of them here.

Given the characterizations of well-behaved intervals presented above, one might wonder how prevalent they are.

\begin{question}\label{ques:prevalence}
What proportion of intervals in the Bruhat order are boolean? What proportion among all intervals of a given size? What proportion among all intervals at a given rank?
\end{question}

One can change ``boolean'' and ``Bruhat'' for other versions of these questions, as well.

In light of the results of Section~\ref{sec:intervals over atoms in Sn}, one might ask questions like the following, with similar variations.

\begin{question}\label{ques:support}
Consider $w \in \symm_n$, and let $T \subseteq S$ be the collection of generators appearing in reduced decompositions of $w$ (i.e., the \emph{support} of $w$). When is it true that $[\s,w]$ is distributive for all $\s \in T$ in the Bruhat order?
\end{question}

In fact, the results above allow us to answer several variations of this question.

\begin{corollary}\label{cor:boolean for all support bruhat}
Consider $w \in \symm_n$ with the Bruhat order. Let $T$ be the support of $w$. Then $[\s,w]$ is boolean (equivalently, a modular lattice, or a distributive lattice) for all $\s \in T$ if and only if $w$ is a product of distinct generators (meaning that $\poi{w}$ is also boolean), or $w = \s_{k+1}\s_k\s_{k+1}$ for some $k$.
\end{corollary}

\begin{proof}
Let $w$ is a product of distinct generators, then $\poi{w}$ is boolean and any interval of it is also boolean. If $w = \s_{k+1}\s_k\s_{k+1}$, it is easy to check that $[\s_i,w]$ is boolean for $i \in \{k,k+1\}$.

Now suppose that $[\s,w]$ is boolean for all $\s \in T$, and recall Theorem~\ref{thm:boolean interval at rank 1} and the limitations it places on repetition in reduced decompositions of $w$. Suppose that there is repetition, say $\cdots \s_{k+1} \cdots \s_k \cdots \s_{k+1} \cdots \in R(w)$. If, in fact, $w = \s_{k+1}\s_k\s_{k+1}$, then we are done.

Suppose, instead, that there is at least one other letter in the reduced decomposition. If $w$ has the form
$$\cdots \s_h \cdots \s_{k+1} \cdots \s_k \cdots \s_{k+1} \cdots,$$
then $[\s_h,w]$ is not boolean, by Theorem~\ref{thm:boolean interval at rank 1}. If, instead, $w$ has the form
$$\cdots \s_{k+1} \cdots \s_h \cdots \s_k \cdots \s_{k+1} \cdots$$
where $\s_h$ does not commute with $\s_{k+1}$ (i.e., $h = k+2$), then $[\s_{k+1},w]$ is not boolean, as demonstrated in Figure~\ref{fig:non-modular lattice}.
\end{proof}

For $n \ge 2$, there are $F_{2n-1} + n-2$ such elements in $\symm_n$: the $F_{2n-1}$ boolean elements (note that the identity trivially satisfies the requirements of Corollary~\ref{cor:boolean for all support bruhat}), together with the $n-2$ permutations with reduced decompositions of the form $\s_{k+1}\s_k \s_{k+1}$ for some $k$.  This is sequence A331347 of \cite{oeis}.

\begin{example}\
\begin{enumerate}
\item[(a)] There are $15$ permutations $w \in \symm_4$ for which $[\s,w]$ is boolean for all $\s$ in the support of $w$. These are marked in Figure~\ref{fig:boolean over support bruhat}. The fact that these are all elements below a given rank is due to the fact that $\symm_4$ has only three generators. For $\symm_n$ with $n > 4$, there will be ranks in which a proper subset of the elements satisfy the requirements of Corollary~\ref{cor:boolean for all support bruhat}. 
\item[(b)] The permutations $3412 = \s_2\s_1\s_3\s_2$ and $4132 = \s_2\s_3\s_2\s_1$ do not satisfy the requirements of Corollary~\ref{cor:boolean for all support bruhat}, and indeed $[\s_2,3412]$ and $[\s_1,4132]$ are not boolean, as shown in Figure~\ref{fig:boolean over support bruhat}.
\end{enumerate}
\end{example}

\begin{figure}[htbp]
\begin{tikzpicture}[scale=1]
\foreach \y in {0,9} {\fill[black] (0,\y) circle (2pt);}
\foreach \x in {-2,0,2} {\foreach \y in {1.5,7.5} {\fill[black] (\x,\y) circle (2pt);};}
\foreach \x in {-4,-2,0,2,4} {\foreach \y in {3,6} {\fill[black] (\x,\y) circle (2pt);};}
\foreach \x in {-5,-3,-1,1,3,5} {\fill[black] (\x,4.5) circle (2pt);}
\draw (0,0) coordinate (1234); \draw (0,9) coordinate (4321);
\draw (-2,1.5) coordinate (1243); \draw (0,1.5) coordinate (1324); \draw (2,1.5) coordinate (2134);
\draw (-2,7.5) coordinate (4312); \draw (0,7.5) coordinate (4231); \draw (2,7.5) coordinate (3421);
\draw (-4,3) coordinate (1423); \draw (-2,3) coordinate (1342); \draw (0,3) coordinate (2143); \draw (2,3) coordinate (3124); \draw (4,3) coordinate (2314);
\draw (-4,6) coordinate (4132); \draw (-2,6) coordinate (4213); \draw (0,6) coordinate (3412); \draw (2,6) coordinate (2431); \draw (4,6) coordinate (3241);
\draw (-5,4.5) coordinate (1432); \draw (-3,4.5) coordinate (4123); \draw (-1,4.5) coordinate (2413); \draw (1,4.5) coordinate (3142); \draw (3,4.5) coordinate (3214); \draw (5,4.5) coordinate (2341);
\draw (1234) -- (1243) -- (1423) -- (1432) -- (4132) -- (4312) -- (4321) -- (3421) -- (3241) -- (2341) -- (2314) -- (2134) -- (1234) -- (1324) -- (1423) -- (4123) -- (4132) -- (4231) -- (4321);
\draw (1243) -- (1342) -- (1432) -- (3412) -- (4312);
\draw (1243) -- (2143) -- (4123) -- (4213) -- (4312);
\draw (1324) -- (1342) -- (3142) -- (4132);
\draw (1324) -- (3124) -- (4123);
\draw (1324) -- (2314) -- (2413) -- (4213) -- (4231);
\draw (2134) -- (2143) -- (2413) -- (3412) -- (3421);
\draw (2134) -- (3124) -- (3142) -- (3412);
\draw (1423) -- (2413) -- (2431) -- (4231);
\draw (1342) -- (2341) -- (2431) -- (3421);
\draw (2143) -- (3142) -- (3241) -- (4231);
\draw (2143) -- (2341);
\draw (3124) -- (3214) -- (4213);
\draw (2314) -- (3214) -- (3412);
\draw (1432) -- (2431);
\draw (3214) -- (3241);
\draw (1234) node[below] {$\emptyset$};
\draw (4321) node[above] {$\s_1\s_2\s_3\s_1\s_2\s_1$};
\draw (2134) node[below right] {$\s_1$}; \draw (1324) node[above] {$\s_2$}; \draw (1243) node[below left] {$\s_3$};
\draw (1342) node[below left] {$\s_2\s_3$}; \draw (1423) node[below left] {$\s_3\s_2$}; \draw (2143) node[below, yshift=-4pt] {$\s_1\s_3$}; \draw (2314) node[below right] {$\s_1\s_2$}; \draw (3124) node[below right] {$\s_2\s_1$};
\draw (1432) node[left] {$\s_2\s_3\s_2$}; \draw (3214) node[right] {$\s_1\s_2\s_1$}; \draw (4123) node[left] {$\s_3\s_2\s_1$}; \draw (2341) node[right] {$\s_1\s_2\s_3$}; \draw (2413) node[left, xshift=-2pt, yshift=2pt] {$\s_1\s_3\s_2$}; \draw (3142) node[right, xshift=2pt, yshift=-2pt] {$\s_2\s_1\s_3$};
\draw (4132) node[left] {$\s_2\s_3\s_2\s_1$}; \draw (4213) node[left, xshift=2pt] {$\s_1\s_3\s_2\s_1$}; \draw (3241) node[right] {$\s_1\s_2\s_1\s_3$}; \draw (2431) node[right] {$\s_1\s_2\s_3\s_2$}; \draw (3412) node[above, yshift=2pt] {$\s_2\s_1\s_3\s_2$};
\draw (4231) node[below, yshift=-5pt] {$\s_1\s_2\s_3\s_2\s_1$}; \draw (4312) node[left] {$\s_2\s_3\s_2\s_1\s_2$}; \draw (3421) node[right] {$\s_1\s_2\s_1\s_3\s_2$};
\foreach \x in {1234,2134,1324,1243,1423,1342,2143,3124,2314,1432,4123,2413,3142,3214,2341} {\draw[ultra thick,red] (\x) circle (2.85pt);}
\end{tikzpicture}
\caption{The Bruhat order of $\symm_4$, where each element is labeled by one of its (possibly many) reduced decompositions. The permutations that form boolean intervals over all generators in their support are marked in red.}\label{fig:boolean over support bruhat}
\end{figure}
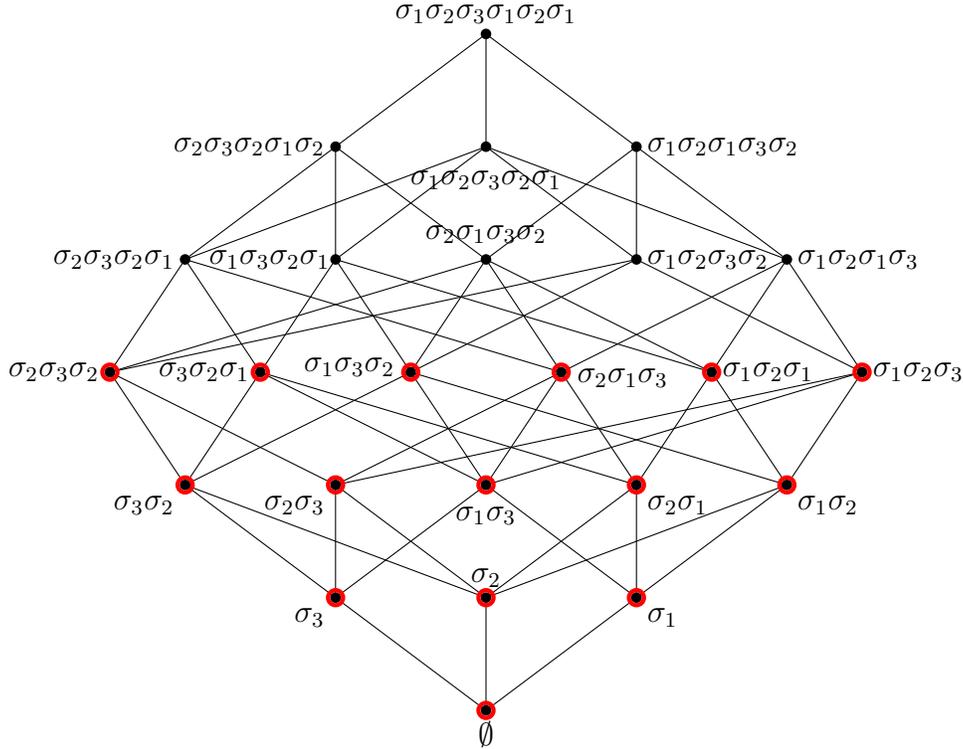

\begin{corollary}\label{cor:boolean for all support weak}
Consider $w \in \symm_n$ with the weak order. Let $T$ be the support of $w$. Then $[\s,w]\weakint$ is boolean for all $\s \in T$ if and only if $w$ is free; i.e., if and only if $\weakpoi{w}$ is boolean.
\end{corollary}

\begin{proof}
If $w$ is free then certainly $w$ has the desired property.

On the other hand, suppose that $[\s,w]\weakint$ is boolean for all $\s \in T$. If $\s_i,\s_{i+1}\in T$, then $w$ cannot be greater than both $\s_i$ and $\s_{i+1}$. Therefore $w$ must be free.
\end{proof}

As described previously, there are $F_{n+1}$ elements in $\symm_n$ satisfying Corollary~\ref{cor:boolean for all support weak}.

Finally, many of the results presented here have strong potential to be generalized to arbitrary Coxeter systems. 

\begin{question}
How can notable interval structures be characterized in the Bruhat and weak orders in Coxeter groups of other types?
\end{question}

\section{Acknowledgements}

I am grateful to many people, including Darij Grinberg for his interest in pattern characterizations of well-behaved principal order ideals, and an anonymous referee for thoughtful comments and suggestions.

\end{document}